  \long\def\hide#1\endhide{}
\documentclass[11pt]{amsart}
\usepackage{eucal}
\usepackage{verbatim} 
\usepackage{ifthen} 
\usepackage{array}
\usepackage{caption}

\usepackage{amsmath,amsthm,amsfonts,amscd,amsxtra,amsopn,amssymb,setspace,verbatim}
\usepackage{stmaryrd}
\usepackage{sseq} 
\usepackage[pdftex]{hyperref} 
\hypersetup{ 
colorlinks,%
citecolor=black,%
filecolor=black,%
linkcolor=black,%
urlcolor=black 
} 

\input xy
\xyoption{all}

\newtheorem{theor}{Theorem}[section]
\newenvironment{theorem}{\begin{theor}}{\end{theor}}
\newtheorem{lemma}[theor]{Lemma}

\newtheorem{proposition}[theor]{Proposition}
\newenvironment{prop}{\begin{proposition}}{\end{proposition}}
\newtheorem{conj}[theor]{Conjecture}

\newtheorem{cond}[theor]{Condition}

\newtheorem{ques}[theor]{Question}

\newtheorem{corollary}[theor]{Corollary}
\newenvironment{cor}{\begin{corollary}}{\end{corollary}}
\newtheorem{state}[theor]{Statement}

\newtheorem{notn}[theor]{Notation}

\newtheorem{exx}[theor]{Example}

\newtheorem{defi}[theor]{Definition}
\newenvironment{defn}{\begin{defi}}{\end{defi}}

\newenvironment{enum}{\begin{enumerate}}{\end{enumerate}}
\newenvironment{itemiz}{\begin{itemize}}{\end{itemize}}

\numberwithin{equation}{section} \numberwithin{theor}{section}

\theoremstyle{definition}
\newtheorem{remark}[theor]{Remark}
\newenvironment{rem}{\begin{remark}}{\end{remark}}

\newcommand\fg{\mathfrak g}
\newcommand\phalf{\frac{p_1}{2}}
\newcommand\dd{\delta}
\newcommand\csu{\frac{i}{2\pi}\tr(\Theta)}

\newcommand\TVP{T^V\!P}
\newcommand\THP{T^H\!P}
\newcommand\bd{\llbracket \delta \rrbracket}
\newcommand\cwfour{\langle \Omega \wedge \Omega \rangle}
\newcommand\csthree{\alpha(\Theta)}

\newcommand\tr{\operatorname{Tr}}

\def\R{\mathbb R}

\def\Z{\mathbb Z}

\def\cH{\mathcal H}

\def\cS{\mathcal S}

\newcommand{\into}{\hookrightarrow}

\newcommand{\Id}{{1\!\!1}}

\def\={:=}
\def\Ker{\operatorname{Ker}}

\def\:{\colon}
\def\iso{\cong}
\def\-1{^{-1}}
\def\o{\circ}

\begin{document}
\title{Harmonic Forms on Principal Bundles}
\author{Corbett Redden}
\address{Mathematics Department, SUNY, Stony Brook, NY 11794-3651, USA}
\curraddr{Michigan State University, East Lansing, MI 48824}
\email{redden@math.msu.edu}


\begin{abstract}
We show a relationship between Chern--Simons 1- and 3-forms and harmonic forms on a principal bundle.  Doing so requires one to consider an adiabatic limit.  For the 3-form case, assume that $G$ is simple and the corresponding Chern--Weil 4-form is exact.  Then, the Chern--Simons 3-form on the princpal bundle $G$-bundle, minus a canonical term from the base, is harmonic in the adiabatic limit.
\end{abstract}

\maketitle

\section{Introduction}
In general, explicitly computing the harmonic forms on a Riemannian manifold is a difficult problem.  However, if one has some symmetries or some extra information, the problem becomes more manageable.  We notice a few interesting things in this paper.  First, if one places a very natural metric on a principal bundle, the harmonic forms have a nice characterization in an adiabatic scaling limit.  In fact, the Chern--Simons 1-form and 3-form naturally arise in this context.  Second, to show that a form is harmonic in the adiabatic limit, one only needs to show that the Laplacian of it tends towards 0 at a fast enough rate in the scaling limit.  One does not need to characterize the harmonic forms on the nose for any particular scaling.

We now proceed to say this more concretely.  All manifolds in this paper are connected, compact, oriented, and without boundary.  Fix a connected compact semisimple Lie group $G$.  Let $P\overset{\pi}\to M$ be a principal $G$-bundle with conection $\Theta$ over the Riemannian manifold $(M,g_M)$.  Choosing some bi-invariant metric $g_G$ on $G$, we use the connection $\Theta$ and base metric $g_M$ to contruct a local product metric $g_P$ on $P$:
\[ g_P \= \pi^* g_M \oplus g_G.\]
The closed manifold $P$ now has a Riemannian metric, and it is natural to investigate the harmonic forms with respect to this metric.

In concrete calculations, however, one observes that this space has no obvious characterization.  In fact, it depends on the choice of bi-invariant metric $g_G$.  Equivalently, the finite-dimensional space of harmonic forms varies as one globally rescales the base metric $g_M$.  For this reason, we introduce a scaling factor in one direction and take a limit.  For any $\dd>0$, define the metric on $P$
\[ g_\dd \= \dd^{-2} \pi^* g_M \oplus g_G. \]
The limit $\dd\to 0$, where the volume of the base is very large with respect to the volume of the fiber, is known as the adiabatic limit and appears in a number of constructions.  Of particular relevance to our situation is the work of Mazzeo--Melrose \cite{MM90}, Dai \cite{Dai91}, and Forman \cite{For95}, which considers the Hodge Laplacian on forms in an adiabatic limit.  While the metric $g_\dd$ and corresponding Laplacian $\Delta_{g_\dd}$ become singular at $\dd=0$, the space of harmonic $k$-forms $\Ker \Delta^k_{g_\dd}$ smoothly extends to $\dd=0$, as proven in all three of the above papers.  We therefore define
\[ \cH^k(P) \= \lim_{\dd\to 0} \Ker \Delta^k_{g_\dd} \subset \Omega^k(P),\]
and note that $\cH^k(P) \iso H^k(P;\R)$ canonically.

To calculate this, we use an interpretation of $\cH^k(P)$ in terms of the Leray--Serre spectral sequence for the fiber bundle $G \into P \to M$.  While this relationship was obtained in \cite{MM90}, \cite{Dai91}, \cite{For95}, it is rather subtle and not commonly known, so we give a brief recount of \cite{For95} in Section \ref{section:ss}.  The key idea is that we do not have to construct a $\dd$-family of $g_\dd$-harmonic forms; instead, we only need to show that a given form is harmonic up to a low-order power of $\dd$.  The required power of $\dd$ is related to when the spectral sequence collapses.  For example, if an $SU(n)$-bundle satisfies $c_2(P)=  0\in H^4(M;\R)$, then the spectral sequence calculating $H^3(P;\R)$ collapses at $N=2$.  That $N=2$, and not higher, makes our calculations possible.

By definition, a principal $G$-bundle has a free right $G$-action, and the metrics $g_\dd$ are all invariant under this action.  This implies that the harmonic forms on $P$ will be contained in the subspace $\Omega^*(P)^G$ of right-invariant forms.  In Section \ref{section:bigradedcomplex}, we give an explicit description of this as a bigraded cochain complex in accordance with the discussion from Section \ref{section:ss}.  We see that it has a very natural description as $\Omega^{i,j}(P)^G \iso \Omega^i(M;\Lambda^j \fg_P^*)$, where $\fg_P$ is the Adjoint bundle associated to $P$.  The exterior derivative decomposes into a Lie algebra derivative $d_\fg$, a covariant connection $d_\nabla$, and a curvature term $\iota_\Omega$.

In Section \ref{section:harmonicforms}, we perform the actual calculations.  The notation is that $\Theta$ is a connection with curvature $\Omega$, and the harmonic forms on $M$ are $\cH^i(M)$.  To begin, in Theorem \ref{theorem:1dim} we calculate the harmonic 1-forms on a $G$-bundle for any connected compact $G$.  For $G=U(n)$, this is shown in Corollary \ref{cor:unitary1dim} to be
\[ \cH^1(P) = \begin{cases} \pi^* \cH^1(M) &\text{if } c_1(P) \neq 0 \in H^2(M;\R)\\
\pi^*\cH^1(M) \oplus \R[\csu - \pi^*h] &\text{if } c_1(P) = 0 \in H^2(M;\R)
\end{cases} \]
where $h \in \Omega^1(M)$ is the unique form such that $dh = \frac{i}{2\pi}\tr(\Omega) = c_1(P,\Theta)$ is the first Chern form and $h \in d^*\Omega^2(M)$.  The cohomology class $c_1(P)$ is the first Chern class and $\csu$ is the Chern--Simons 1-form.  We then see in Proposition \ref{prop:baseharm} that when $G$ is semisimple, the adiabatic-harmonic 1 and 2-forms are
\[ \cH^1(P) = \pi^* \cH^1(M), \qquad  \cH^2(P) = \pi^* \cH^2(M).\]
Finally, Theorem \ref{thm:harmonic3forms} calculates the adiabatic-harmonic 3-forms when $G$ is simple.  We denote the Chern--Simons 3-form by $\alpha(\Theta)\in \Omega^3(P)$, and the Chern--Weil 4-form by $\cwfour \in \Omega^4(M)$.  For $G=SU(n)$ with $n\geq 2$, the harmonic forms in the adiabatic limit are 
\[ \cH^3(P) = \begin{cases}\pi^* \cH^3(M) & c_2(P)\neq 0 \in H^4(M;\R) \\
\pi^* \cH^3(M)  \oplus \R[\csthree - \pi^*h]  & c_2(P)=0 \in H^4(M;\R) \end{cases} \]
where $h \in \Omega^3(M)$ is the unique form such that $dh = \cwfour = c_2(P,\Theta)$ and $h \in d^*\Omega^4(M)$.  In particular, the Chern--Simons form $\alpha(\Theta)$ is in $\cH^3(P)$ precisely when the second Chern form $\cwfour$ is identically 0.  When $n=3$ or $n \geq 5$, the above statement also holds with $SU(n)$ replaced by $Spin(n)$ and $c_2$ replaced by $\phalf$.

We briefly discuss one use for these results which appears in \cite{Redden09}.  First consider the subset of forms with \textit{integral} periods and define $\cH^3_\Z(P) \subset \cH^3(P)$ as the image of $H^3(P;\Z) \to H^3(P;\R) \overset{\iso}\to \cH^3(P)$.  In the situation where $G=Spin(n)$ with $n \geq 5$, there is standard way to choose $\langle \cdot, \cdot \rangle$ so that $\alpha(\Theta)$ restricts to the generator of $\cH^3_\Z(Spin(n)) \iso \Z$.  It follows from Theorem \ref{thm:harmonic3forms} that if $\phalf(P)=0\in H^4(M;\Z)$, then
\[ \cH^3_\Z(P) = \pi^* \cH_\Z^3(M) \oplus \Z[\csthree - \pi^*H] ,\]
where $H$ is unique up to the addition of an element from $\cH_\Z^3(M)$.  The form $H \in \Omega^3(M)$ satisfies $dH=\cwfour$ and $d^*H=0$, but in general it has a non-trivial harmonic component due to the fact that $\alpha(\Theta)-\pi^*h$ will not have integral periods.  The non-integrality of $\alpha(\Theta)-\pi^*h$ is precisely because, modulo $\Z$, this inegrates to give Chern--Simons numbers on 3-cycles, which are often non-zero \cite{CS74}.

In \cite{Redden09}, we discuss the notion of ``string structures" on a principal bundle and see that a string structure picks out a canonical cohomology class $\cS \in H^3(P;\Z)$ that restricts fiberwise to the standard generator of $H^3(Spin(n);\Z)$.  By the results of this paper, the canonical differential form $[\cS]\in \cH^3_\Z(P)$ representing $\cS$ will be
\[ [\cS] = \alpha(\Theta) - \pi^*H \in \Omega^3(P),\]
thus specifying a \textit{canonical} $H=H_{g_M, \Theta, \cS}$.  Choosing a string structure can be thought of as choosing an $H$ such that $\alpha(\Theta)-\pi^*H \in \cH^3_\Z(P).$  An important consequence is that the form $H$ lifts the values of the associated differential character from $\R/\Z$ to $\R$.

Finally, the author would like to thank Stephan Stolz for all of his encouragement and insight.  He is also grateful to Brian Hall, Liviu Nicolaescu, and Bruce Williams for numerous helpful comments and discussions.

\section{Hodge Theory and Spectral Sequences}\label{section:ss}

The idea that the harmonic forms on a fiber bundle are related to a spectral sequence was developed in \cite{MM90} and \cite{Dai91} and expanded on in \cite{For95}.  In particular, the spectral sequence structure was described in \cite{Dai91} and made even more explicit  in \cite{For95}.  In the following, we summarize Forman's treatment to make this paper self-contained.  However, the reader should be warned that we use the bigrading notation (horizontal, vertical), which is opposite that of Forman.  Before getting to the adiabatic spectral sequence, we quickly review Hodge cohomology and spectral sequences.

\subsection{Hodge Theory}Consider a Riemannian manifold $(M,g_M)$ that is compact with no boundary.  The operator $d$ on differential forms naturally gives rise to de Rham cohomology.  However, a de Rham cohomology class only determines an equivalence class of forms representing the cohomology class.  Choosing a Riemannian metric and orientation produces the Hodge star $*: \Lambda^k TM^* \to \Lambda^{n-k}TM^*$, the adjoint operator $d^* = (-1)^{n(k+1)+1}*d*$, and the Hodge Laplacian
\[ \Delta^k \= (d+d^*)^2 = dd^* + d^* d: \Omega^k(M) \to \Omega^k(M).\]
Forms in the kernel of $\Delta$ are called harmonic, and we use the notation
\[ \cH^k(M) \= \Ker \Delta^k  = \Ker d \cap \Ker d^* \subset \Omega^k(M).\]
Elements in $\cH^k(M)$ are called harmonic, and we refer to $\cH^k(M)$ as the Hodge cohomology of the complex $(\Omega^*(M), d, d^*)$.  The main theorem of Hodge theory is the orthogonal decomposition of forms
\begin{equation}\label{eq:hodgedecomp}
\Omega^k(M) = d\Omega^{k-1} \oplus d^* \Omega^{k+1} \oplus \cH^k(M).
\end{equation}
In particular, the finite-dimensional vector space $\cH^k(M)$ of harmonic forms is \textit{canonically} isomorphic to the de Rham cohomology $H^k(M;\R)$.  Furthermore, there are natural isomorphisms between exact and coexact forms
\begin{equation}\label{eq:ddstariso} \xymatrix{ &d^*\Omega^{k+1} \ar@/_/[r]_{d} &d\Omega^k(M). \ar@/_/[l]_{d^*} } \end{equation}

For generic Riemannian manifolds, explicitly describing harmonic forms is extremely difficult.  However, on a compact connected Lie group with bi-invariant metric, the situation is much easier.  The group acts on itself by isometries, and it follows that any harmonic form must be invariant under the left and right group actions.  In fact, it can be shown that any bi-invariant form must be closed, leading to Proposition \ref{prop:liealgharmonic}.  Restricting to the subcomplex of left-invariant forms, we obtain the standard complex used in Lie algebra cohomology.  The adjoint $d^*$ maps left-invariant forms to left-invariant forms, and so descends to an adjoint at the Lie algebra level:
\begin{equation}\label{eq:liealgcoh} \xymatrix{ 0 \ar[r] &\R \ar[r]^0\ar@/_1pc/[l]  &\fg^* \ar[r]^{d_\fg}\ar@/_1pc/[l]_{0} &\Lambda^2\fg^* \ar[r]^{d_\fg} \ar@/_1pc/[l]_{d_\fg^*} &\cdots \ar[r]^{d_\fg} \ar@/_1pc/[l]_{d_\fg^*} &\Lambda^{n-1}\fg^* \ar[r]^{0} \ar@/_1pc/[l]_{d_\fg^*} &\Lambda^n \fg^* \ar[r] \ar@/_1pc/[l]_0 &0.\ar@/_1pc/[l]. }\end{equation}
The harmonic forms $\cH^k(\fg) = \Ker (d_\fg d^*_\fg + d^*_\fg d_\fg)$ are naturally isomorphic to $\cH^k(G)$ and are described by the following standard result (see, for instance, 7.4 of \cite{MR1435504}).
\begin{prop}\label{prop:liealgharmonic}For connected compact Lie groups with bi-invariant metric, the harmonic forms are exactly the bi-invariant forms.  Equivalently, for Lie algebras with $Ad$-invariant metric, the harmonic forms in \eqref{eq:liealgcoh} are exactly the $Ad$-invariant forms.
\end{prop}

\subsection{Leray--Serre spectral sequence}\label{subsection:SerreSS}
The Leray--Serre cohomology spectral sequence for a fibration $F \into P \to M$ consists of: 
\begin{itemiz}
\item a sequence of ``pages" or cochain complexes $E^p_0, E^p_1, E^p_2, \cdots, E^p_\infty$, equipped with a bigrading $E^p_K = \displaystyle\bigoplus_{i+j=p} E^{i,j}_K$,
\item differentials $d_K: E^p_K \to E^{p+1}_K$ such that $H^p(E^*_K, d_K) = E^p_{K+1}$.  With respect to the bigrading, $d_K: E^{i,j}_K \to E^{i+K, j-K+1}_K$,
\item $E^{i,j}_2 \iso H^i(M; H^j(F;H)),$ where $H$ is an abelian group (if $\pi_1(M)$ acts trivially on $P$),
\item $E^p_\infty$ gives $H^p(M;H)$ in the sense that $E^{n, p-n}_\infty $ is isomorphic to the quotient $F^p_n / F^p_{n+1}$ of a filtration
$0 \subset F^p_p \subset \cdots \subset F^p_0 = H^p(P;H)$.
\end{itemiz}

In other words, it relates the cohomology of the total space of a fibration with the cohomology of its base and fiber.  This allows for powerful calculations using formal methods.  The primary example we deal with is the case where $P$ is a principal bundle over a manifold $M$.  We note that if $G$ is connected, then $\pi_1(BG)=0$.  Since any $G$-bundle is isomorphic to a pullback of the universal bundle $EG \to BG$, $\pi_1(M)$ acts trivially on the fibers of $P$.  For examples of spectral sequence calculations, see Section \ref{section:harmonicforms}.

\subsection{Adiabatic spectral sequence}
While the results of \cite{For95} apply to more general situations, we restrict to a specific geometric hypothesis to simplify matters.  Suppose that $F \into P \overset{\pi}\to M$ is a fiber bundle of compact manifolds such that $\pi$ is a Riemannian submersion; i.e. $(P,g_P)$ and $(M,g_M)$ are Riemannian manifolds, the map $\pi$ is a Riemannian submersion, and the fibers $\pi^{-1}(x) \subset P$ are Riemannian submanifolds isometric to a fixed Riemannian manifold $(F,g_F)$.  Thinking of the fibers $F$ as the vertical direction and the base $M$ as the horizontal direction, the tangent bundle $TP$ decomposes as
\[ TP = \THP \oplus \TVP. \]
Here, \[ \TVP \= \Ker \pi_* \subset TP, \text{ and } \THP \= TP \ominus \TVP \subset TP\]
are distributions of vertical and horizontal vectors respectively ($\THP$ is defined by taking the orthogonal complement of $\TVP$ using the metric $g_P$).  The metric on $P$ is then the direct sum of metrics on the horizontal and vertical distributions:
\[ g_P = g_M \oplus g_F.\]
Define the 1-parameter family of metrics $g_\delta$, where $\delta >0,$ by
\begin{equation}\label{eq:gdelta} g_\delta \= \delta^{-2}g_M \oplus g_F. \end{equation}
The limit $\delta \to 0$ is called the \textit{adiabatic limit} and can be thought of as making the base very large relative to the fibers.

The orthogonal decomposition of $TP$ induces a bigrading on the differential forms
\[ \Omega^p(P) = \bigoplus_{i+j=p} \Omega^{i,j}(P), \quad \Omega^{i,j}(P) = C^\infty(P, \Lambda^i \THP^* \otimes \Lambda^j \TVP^*); \]
The exterior derivative $d$ and its adjoint $d^*$, under the fixed metric $g_P$, decompose as
\[ d = d^{0,1} + d^{1,0} + d^{2,-1}, \quad d^{a,b}:\Omega^{i,j}(P) \to \Omega^{i+a,j+b}(P), \]
\[ d^* = d^{0,1*} +  d^{1,0*} + d^{2,-1*}, \quad d^{a,b*}: \Omega^{i,j}(P) \to \Omega^{i-a, j-b}(P). \]
The component $d^{-1,2}=0$ due to the integrability of $\TVP$.  To effectively deal with $g_\delta$, introduce the isometry for any $\dd>0$
\begin{align*}
\rho_\dd \: (\Omega^{i,j}(P), g_\dd) &\to (\Omega^{i,j}(P), g_P) \\
\phi &\mapsto \dd^i \phi. \\
\end{align*}
Conjugating by $\rho_\dd$ gives rise to a 1-parameter family of operators $d_\dd, d^*_\dd,$ and $L_{g_\dd}$ on the fixed inner product space $(\Omega^*(P), g_P)$.  These evidently take the form
\begin{equation}
\label{eq:ddelta}\begin{split}d_\dd \={}& \rho_\dd d \rho_\dd \-1 = d^{0,1} +  \dd d^{1,0} + \dd^2 d^{2,-1},\\
d^*_\dd \={}& \rho_\dd d^{*_{g_\dd}} \rho_\dd \-1 = d^{0,1*} +  \dd d^{1,0*} + \dd^2 d^{2,-1*},\\
L_{g_\dd} \={}& \rho_\dd \Delta_{g_\dd} \rho_\dd \-1 = d_\dd d^*_\dd + d^*_\dd d_\dd.
\end{split}\end{equation}

In addition to fixing the inner product, the isometry $\rho_\dd$ produces the above factors of $\dd$ which naturally give the spectral sequence structure.  We first define the terms $E^{i,j}_K$, and then we define the differentials and explain why this is a spectral sequence.

\begin{defn}\label{defn:epages}
\begin{align*}E^{i,j}_K \= \{ &\omega \in \Omega^{i,j}(P) \>|\> \exists \> \omega_1, \ldots, \omega_l \in \Omega^{i+j}(P)\> \text{with} \\
&d_\dd (\omega + \dd \omega_1 + \cdots + \dd^l \omega_l) \in \dd^K\Omega^{i+j+1}(P)[\dd] \\
&d^*_\dd (\omega + \dd \omega_1 + \cdots + \dd^l \omega_l) \in \dd^K\Omega^{i+j-1}(P)[\dd] \} \\
E^{i,j}_\infty \= {}{}&\bigcap_{K} E^{i,j}_K
\end{align*}\end{defn}
While the above definition seems formal, there is a nice geometric interpretation.  The polynomial $\omega + \dd \omega_1 + \cdots + \dd^l\omega_l$ should be thought of as the Taylor approximation at $\dd=0$ of a function
\[ \widetilde{\omega} \in C^\infty([0,1], \Omega^{i+j}(P)),\]
and elements $\omega \in E^{i,j}_K$ as boundary values of sections $\widetilde{\omega}$ such that as $\delta\to 0$,
\[ d_\dd \widetilde{\omega} = 0 + O(\dd^K);\> d^*_\dd \widetilde{\omega} = 0 + O(\dd^K). \]

By definition, $\Omega^{i,j}(P) = E^{i,j}_0 \cdots \supseteq E^{i,j}_K \supseteq E^{i,j}_{K+1} \supseteq \cdots \supseteq E^{i,j}_\infty$.  To define the differentials, first let \[ \pi_K:\Omega^*(P)\to E_K \subset \Omega^*(P)\] denote the orthogonal projection using the metric $g_P$.  To any $\omega \in E_K$,  there exists a non-unique polynomial $\omega_\dd = \omega + \dd \omega_1 +\cdots + \dd^l\omega_l$ satisfying Definition \ref{defn:epages}.  Define
\[  d_K \omega_\dd \= \lim_{\dd \to 0} \dd^{-K} d_\dd \omega_\dd = \lim_{\dd \to 0} \dd^{-K} d_\dd (\omega + \dd\omega_1 + \dd^2\omega_2 + \cdots) \in \Omega^*(P).\]
Though this depends on the choice of polynomial $\omega_\delta$,
\[ \pi_K d_K : E_K \to E_K \] is well-defined.  The same procedure defines \[ \pi_K d^*_K : E_K \to E_K \]  and the second order operator
\[\vartriangle_K \= \left(\pi_K d_K\right)  \left(\pi_K d^*_K \right) + \left(\pi_K d^*_K \right) \left(\pi_K d_K \right).\]

\begin{theorem}[Theorem 2.5 \cite{For95}] $\>$

\begin{enum}
\item $(\pi_K d_K )^2 = (\pi_K d^*_K )^2 = 0$.
\item $\Ker \vartriangle_K = \Ker (\pi_K d_k )\cap \Ker (\pi_K d^*_K )= E_{K+1}$.
\item $\pi_Kd_K : E_K^{i,j} \to E_K^{i+K, j-K +1}$.
\end{enum}\end{theorem}

Therefore, the complex $\{E^{i,j}_K,\pi_K d_K\}$ is a bi-graded spectral sequence.  However, instead of taking the ordinary cohomology and dealing with equivalences classes of cocycles, $E^\bullet_{K+1}$ is obtained as the Hodge cohomology of the complex $E^\bullet_{K}$.  The cochains of each subsequent page in the spectral sequence are still represented by differential forms.  By Corollary 4.4 in \cite{For95}, the spectral sequence is isomorphic to the standard Leray--Serre spectral sequence associated to the fibration $F \into P\overset{\pi}\to M$.  Let $N = N(p)$ denote the page where the portion of the spectral sequence calculating $H^p(P;\R)$ collapses; i.e. N is such that
\[ E^p_{N(p)-1} \neq E^p_{N(p)} = \cdots = E^p_\infty. \]

\begin{prop}\label{prop:basis}
For $\omega \in E^p_\infty$, there exists a unique formal power series
\[ \omega_\dd = \omega + \dd \omega_1 + \dd^2 \omega_2 + \cdots \in \Omega^p(P) \bd \]
such that $\omega_i \in E^\perp_\infty$ for all $i \geq 1$ and formally $d_\dd \omega_\dd = d^*_\dd \omega_\dd = 0$.

Furthermore, for any $\omega_\dd \in \Omega^p(P)[\dd]$ satisfying
\[\omega_l \in E^\perp_\infty \;(1\leq l \leq L) \text{ and } d_\dd \omega_\dd, d^*_\dd \omega_\dd \in \dd^{N+L}\Omega^*(P)[\dd], \]  
the terms $\omega_l$ for $l \leq L$ are the terms in the unique power series above.
\begin{proof}
The first part of the proposition follows from the second by considering arbitrarily large $L$.  To prove the second part, suppose $\omega + \dd \omega_1 + \cdots$ and $\omega + \dd \phi_1 + \cdots$ satisfy the assumptions above.  Then, we have that 
\begin{align*}
d_\dd (\omega + \dd \omega_1 + \dd^2 \omega_2 + \cdots) - d_\dd (\omega + \dd \phi_1 + \dd^2 \phi_2 + \cdots)& \\
=d_\dd (\dd (\omega_1 - \phi_1) + \dd^2 (\omega_2 - \phi_2) + \cdots)& \in \dd^{N+L}\Omega^{p+1}(P)[\dd],\\
d^*_\dd (\omega + \dd \omega_1 + \dd^2 \omega_2 + \cdots) - d^*_\dd (\omega + \dd \phi_1 + \dd^2 \phi_2 + \cdots)& \\
= d^*_\dd (\dd (\omega_1 - \phi_1) + \dd^2 (\omega_2 - \phi_2) + \cdots)& \in \dd^{N+L}\Omega^{p-1}(P)[\dd]
\end{align*}
Hence, $\omega_1 - \phi_1 \in E^p_{N+L-1}= E^p_\infty$, but we know that $\omega_1 - \phi_1 \in (E^{p}_\infty)^\perp$.  Therefore, $\omega_1 = \phi_1$, and we continue inductively to show uniqueness of the higher order terms for $l \leq L$.
\end{proof} \end{prop}

\subsection{Relation to harmonic forms}\label{subsection:adiabaticharmonic}
We now wish to discuss the convergence of the spectral sequence in relation to the adiabatic limit of harmonic forms.  As a reminder, the Laplacians $\Delta^p_{g_\dd}$ and $L^p_{g_\dd}$ are defined for all $\dd>0$ and have finite-dimensional kernels naturally isomorphic to $H^p(P;\R)$.  Define the spaces $\Ker \Delta^p_0$ and $\Ker L^p_0$ by
\begin{equation}\begin{aligned} \Ker \Delta^p_0 \= \lim_{\dd\to 0}\Ker \Delta^p_{g_\dd} = \{ \omega \in \Omega^p(P) \>|\>& \exists \> \widetilde{\omega} \in C^\infty([0,1],\Omega^p(P)); \\
& \widetilde{\omega}(0)=\omega, \>  \Delta^p_{g_\dd}\widetilde{\omega}(\dd) = 0 \>\forall \dd > 0\} \end{aligned}\end{equation}
and
\begin{equation}\begin{aligned} \Ker L^p_0 \= \lim_{\dd\to 0}\Ker L^p_{g_\dd} = \{ \omega \in \Omega^p(P) \>|\>& \exists \> \widetilde{\omega} \in C^\infty([0,1],\Omega^p(P)); \\
& \widetilde{\omega}(0)=\omega, \>  L^p_{g_\dd}\widetilde{\omega}(\dd) = 0 \>\forall \dd > 0\} \end{aligned}\end{equation}
In other words, an element $\omega \in \Ker L^p_0$ if and only if it is the limit of a 1-parameter family of $L_{g_\dd}$-harmonic forms.  Suppose $\widetilde{\omega}(\dd)$ is such a smooth function on $[0,1]$ with values in $\Ker L_{g_\dd}$ for all $\dd > 0$.  It follows that the Taylor series at $\dd=0$, denoted $\omega_\dd \in \Omega^p(P)\bd$, is formally $L^p_{\dd}$ harmonic.  In other words,
\[ \widetilde{\omega}(\dd) \underset{\dd=0}\sim \omega_\dd = \omega + \dd \omega_1 + \dd^2 \omega_2 + \cdots \]
and
\[ L_\dd ( \omega + \dd\omega + \dd^2 \omega_2 + \cdots ) = 0.\]
This gives a map of finite-dimensional vector spaces $\Ker L^p_0 \to E^p_\infty$.  In fact, this is an equality.  Forman proves this by a careful analysis of the eigenvalues of $L_{g_\dd}$.  If the eigenvalues of $L^p_{g_{\dd}}$ are ordered, then the number of eigenvalues $\lambda_i(\dd)$ (counted with multiplicity) such that $\lambda_i(\dd) = 0 + O(\dd^{2K})$ is equal to the dimension of $E^p_K$ (Cor. 5.14 \cite{For95}).  This, combined with the knowledge of $E^p_K$ as a spectral sequence, implies that if $\lambda_i(\dd) = 0 + O(\dd^{2N})$, then $\lambda^i(\dd) = 0 + O(\dd^K)$ for all $K$, and therefore corresponds to an element of $E_\infty$.  The spectral sequence interpretation of $E^p_\infty$ implies that $\dim E^p_\infty = \dim H^p(P;\R)$, and standard Hodge theory implies that the number of 0-eigenvalues for all $\dd>0$ is the dimension of $H^p(P;\R)$.  Therefore, 
\begin{align*} \dim H^p(P;\R) = \dim E^p_\infty = \#\{ \lambda_i(\dd) \> | \> \lambda_i(\dd) = 0 + O(\dd^\infty)\} \\
\geq \# \{ \lambda_i(\dd) \>| \> \lambda_i(\dd) = 0\>\forall \>\dd >0\} = \dim H^p(P;\R). \end{align*}
If $\lambda_i(\dd)$ is an eigenvalue such that $\lambda_i(\dd)= 0+O(\dd^N)$, then $\lambda_i$ is identically 0.

Let $\Pi_{\Ker L_{g_\dd}}$ denote the orthogonal projection onto $\Ker L_{g_\dd}$.  Suppose that $\omega \in E^k_\infty$, and $\omega_\dd = \omega + \dd \omega_1 + \ldots + \dd^l \omega_l$ with $\omega_i \in E^\perp_\infty$ for $i\geq 1$ is such that both \[ d_\dd \omega_\dd = 0 + O(\dd^{N+M}),\text{ and } d^*_\dd \omega_\dd = 0 + O(\dd^{N+M}).\]
Then, Theorem 5.21 in \cite{For95} implies 
\begin{equation}\label{eq:projLharm} \left(\Pi_{\Ker L_{g_\dd}} - \Id \right) \omega_\dd = 0 + O(\dd^M).\end{equation}
Therefore, given any $\omega \in E_\infty$, there exists
\[ \widetilde{\omega}(\dd) \= \Pi_{\Ker L_{g_\dd}} \omega_\dd \in C^M([0,1], \Ker \Delta_{g_\dd}) \] such that the first $M$-terms in the Taylor expansion at the boundary coincide with the polynomial $\omega_\dd$.  Considering arbitrarily large polynomials of the form in Proposition \ref{prop:basis} implies the following theorem.

\begin{theorem}[Cor 5.18 \cite{For95}]\label{thm:L0smooth} 
The space of $L_{g_\dd}$-harmonic forms extends smoothly to $\dd=0$, and 
\[\lim_{\dd\to 0} \Ker L^p_{g_\dd} = E^p_\infty \subset \Omega^p(P).\]
\end{theorem}

From here, it easily follows that the finite-dimensional spaces $\cH^p_{g_\dd}(P)$ of $\Delta_{g_\dd}$-harmonic forms extends smoothly to $\dd=0$.  To see this explicitly, we know that the spaces of formally harmonic $\Delta_\dd$ and $L_\dd$ power series are isomorphic over the ring of formal Laurent series in $\dd$.  To a given $\omega \in E^p_\infty$, we can associate the unique power series $\omega_\dd \in \Omega^p(P)\bd$ as described in Proposition \ref{prop:basis}.  Then, for some finite $i$,
\[ \rho_\dd^{-1} \omega_\dd = \dd^{-i} \psi_\dd \in \dd^{-i}\Omega^p(P)\bd,\]
where $\psi_\dd = \psi_0 + \dd \psi_1 + \cdots$ is formally harmonic.  The projection operators $\Pi_{\Ker L_{g\dd}}$ and $\Pi_{\Ker \Delta_{g\dd}}$ are related by 
\[ \Pi_{\Ker \Delta_{g\dd}} = \rho_\dd^{-1} \Pi_{\Ker L_{g\dd}} \rho_\dd. \]  Consider the polynomial $\psi_0 + \ldots \dd^l \psi_l$ obtained from truncating $\psi_\dd$ at some $l$ such that
\begin{align*} d(\psi_0 + \dd \psi_1 + \cdots + \dd^l \psi_l) &= 0 + O(\dd^{N+k+i+1}), \\
d^{*g_\dd}(\psi_0 + \dd \psi_1 + \cdots + \dd^l \psi_l) &= 0 + O(\dd^{N+k+i+1}). \end{align*}
It will then satisfy
\begin{align*}\label{eq:projdeltaharm}
&(\Pi_{\Ker \Delta_{g_\dd}} -\Id) (\psi_0 + \ldots \dd^l \psi_
l) = \rho_\dd^{-1} (\Pi_{\Ker L_{g_\dd}} - \Id) \dd^i(\omega + \dd \omega_1 + \ldots ) \\
& = \dd^i \rho_\dd^{-1} (0 + O(\dd^{k+1})) = 0 + O(\dd)
\end{align*}
The equalities above are obtained from \eqref{eq:projLharm}, the characterization of $\omega_\dd$, and the fact that $\rho_\dd^{-1}$ divides by powers of $\dd$ at most $k$.

Therefore, $\psi_0$ is the limit of a continuous 1-parameter family of $\Delta_{g_\dd}$-harmonic forms.  In fact, the argument can be repeated using higher-order polynomial approximations of $\psi_\dd$ to show the following.

\begin{theorem}[Cor 5.22 \cite{For95}, Cor 18 \cite{MM90}]\label{thm:adiabaticharmonic}The space $\cH_{g_\dd}^p(P)$ of $\Delta_{g_\dd}$-harmonic forms extends smoothly to $\cH^p(P)$ at $\dd=0$.
\end{theorem}

\begin{prop}\label{prop:EtoDelta}Suppose $\omega_0 + \dd \omega_1 + \cdots + \dd^k\omega_k + O(\dd^{k+1}) \in \Omega^k(P)\bd$ is a formally $L_\dd$-harmonic power series that is in the image of $\rho_\dd$; i.e. applying $\rho_\dd^{-1}$ produces no negative powers of $\dd$.  Then, applying $\rho_\dd\-1$ and taking the constant term gives an element of $\Ker \Delta^k_0$:
\[ \left( \rho_\dd \-1 (\omega_0 + \dd\omega_1 + \cdots +\dd^k \omega_k) \right)_{\dd=0} = \]
\[=  \omega_0^{0,k} + \omega^{1,k-1}_1 + \omega^{2,k-2}_2 + \cdots + \omega^{k,0}_k \in \cH^k(P) \subset \Omega^k(P). \]
\begin{proof}
Due to Theorem \ref{thm:L0smooth}, there exists a family of $L_{g_\dd}$-harmonic forms
\[ \widetilde{\omega} \in C^\infty([0,1], \Ker L_{g_\dd} ) \]
such that, close to $\dd=0$,
\[ \widetilde{\omega}(\dd) \underset{\dd=0}\sim \omega_0 + \dd\omega_1 + \cdots + \dd^k\omega_k + O(\dd^{k+1}).\]
Under the isometry $\rho_\dd\-1$,
\[ \rho_\dd\-1\widetilde{\omega} \in C^\infty([0,1],\Ker \Delta^k_{g_\dd}),\]
and
\[ \rho_\dd\-1 \widetilde{\omega} \underset{\dd=0}\sim \rho_\dd\-1 \left(\omega_0 + \dd\omega_1+\cdots + \dd^k\omega_k + O(\dd^{k+1}) \right).\]
The isometry $\rho_\dd\-1$ divides at most by $\dd^k$.  Therefore,
\begin{align*} \rho_\dd\-1 \left(\omega_0 + \dd\omega_1 + \cdots + \dd^k\omega_k + O(\dd^{k+1})\right) &= \\
\left( \omega^{0,k}_0 + \omega^{1,k-1}_1 + \omega^{2,k-2}_2 + \cdots + \omega^{k,0}_k\right) &+ O(\dd).\end{align*}
Consequently,
\begin{align*} \left( \rho_\dd\-1\widetilde{\omega} \right) (0) &= \left(\rho_\dd \-1 (\omega + \dd\omega_1 + \cdots + \dd^k \omega_k)\right)_{\dd=0}\\
&= \omega^{0,k} + \omega^{1,k-1}_1 + \omega^{2,k-2}_2 + \cdots + \omega^{k,0}_k \in \Ker \Delta^k_0.
\end{align*}
\end{proof}\end{prop}

\begin{cor}\label{cor:Ei0toDelta}There is an inclusion of vector spaces (as subspaces of $\Omega^*(P)$)
\[ E^{i,0}_\infty \subset \cH^i(P) \subset \Omega^i(P). \]
\begin{proof}
Let $\omega\in E^{i,0}_\infty$.  By definition, there exists a power series of the form
\[ \omega + O(\dd) \in \Ker L_\dd,\]
and consequently
\[ \dd^k\omega + O(\dd^{k+1}) \in \Ker L_\dd.\]
Applying Proposition \ref{prop:EtoDelta}, we see that
\[ \left(\rho_\dd\-1(\dd^k\omega)\right)_{\dd=0} = \omega \in \Ker \Delta^k_0.\]
\end{proof}
\end{cor}

\section{Right-Invariant Forms on a Principal Bundle}\label{section:bigradedcomplex}
We want to use the adiabatic spectral sequence to analyze harmonic forms on a principal bundle.  If the metric on $P$ is invariant under the free right $G$-action, then the harmonic forms will be contained in the subcomplex of right-invariant forms.  We proceed to characterize these and relate them to the bigraded complex considered in Section \ref{section:ss}.  Once this characterization is complete, the calculations in Section \ref{section:harmonicforms} follow with relative ease.  In particular, a number of maps in the bigraded complexes \eqref{eq:complex} and \eqref{eq:dualcomplex} are 0, which leads to a great deal of simplification.  Much of the following language of principal bundles, as well as the later description of the Chern--Simons 3-form, models that used in \cite{Fre95}.  More explicit details and proofs for the following statements are given in \cite{Redden06}.

\subsection{Vertical Distribution}Let $G$ be a compact Lie group, and define $\fg$ to be the associated Lie algebra of left-invariant vector fields on $G$.  A principal $G$-bundle $P\overset{\pi}\to M$ is a manifold $P$ with a free right $G$-action such that $P \overset{\pi}\to P/G=M$ is the natural quotient.  This implies that each fiber
\[ P_x \= \pi^{-1}(x) \subset P \]
is a right $G$-torsor ($P_x$ has a free and transitive right $G$-action).  The manifold $P_x$ is diffeomorphic to $G$, but only after choosing an initial point $p\in P_x$.  A choice of point $p\in P_x$, gives the (right) $G$-equivariant map
\begin{align*} \tau_p: P_x &\to G \\
p\cdot g &\mapsto g.
\end{align*}
If we choose a different $p'\in P_x$, we see that $p = p'\cdot h$ for a unique $h\in G$, and that
\[ \tau_p' = L_h \o \tau_p,\]
where $L_h:G\to G$ is left multiplication by $G$.

Therefore, any left-invariant structure on $G$ can naturally be placed on a right $G$-torsor.  For example, let $\theta \in \Omega^1(G;\fg)$ be the Maurer-Cartan 1-form, defined by associating to any vector its left-invariant extension.  Then,
\begin{align*}
L_g^* \theta &= \theta, \\
R_g^* \theta &= Ad_{g^{-1}} \theta,
\end{align*}
so $\theta$ pulls back to a well-defined $\fg$-valued 1-form on $P_x$, which we also denote as $\theta$, and satisfies $R_g^*\theta = Ad_{g^{-1}}\theta$.  This gives a natural isomorphism
\[ TP_x \overset{\theta}\to P_x \times \fg \]
that is equivariant with respect to the right $G$-action on $TP_x$, $P_x$, and $\fg$ ($G$ acts on $\fg$ from the right by the inverse Adjoint representation).  Consequently, a right-invariant vector field on $P_x$ is equivalent to a function $P\overset{v}\to \fg$ such that $v(pg)=Ad_{g^{-1}}v(p)$ for all $g\in G$.  More concisely, it is an element of the vector space
\begin{equation}\label{eq:torsorisos} \fg_{P_x} \= P_x \times_{Ad} \fg = P_x \times \fg / \left((p, v)\sim (pg, Ad_{g^{-1}}v) \right).\end{equation}

The Lie bracket $[\cdot, \cdot]:\fg\times \fg \to \fg$ is $Ad$-equivariant because $Ad_g:\fg \to \fg$ is a Lie algebra automorphism.  Therefore, the product bracket on $(P_x \times \fg)\times (P_x \times \fg) \to P_x \times \fg$ descends to a Lie bracket
\begin{equation}\label{eq:rightbracket} \fg_{P_x} \times \fg_{P_x} \overset{[ \cdot, \cdot]}\longrightarrow \fg_{P_x}.\end{equation}
Unpackaging these isomorphisms shows that the bracket \eqref{eq:rightbracket} is the natural Lie bracket of right-invariant vector fields on $P_x$.  We can then consider the cochain complex of right-invariant forms $\{ \Lambda^\bullet \fg_{P_x}^*, d_{\fg}\},$ where
\begin{equation}\label{eq:dg} d_\fg \psi (X_0, \ldots, X_k) = \sum_{i<j} (-1)^{i+j} \psi( [X_i,X_j], \ldots, \widehat{X_i}, \ldots, \widehat{X_j}, \ldots, X_k),
\end{equation}
for $\psi\in \Lambda^k\fg_{P_x}^*, X_i \in \fg_{P_x}$.  A classical theorem of Chevalley and Eilenberg \cite{CE48} shows that the inclusion of cochain complexes
\[ \{ \Lambda^\bullet \fg_{P_x}^*, d_{\fg} \} \iso \{ \Omega^\bullet(P_x)^G, d\} \into \{ \Omega^\bullet(P_x), d\} \] induces an isomorphism on the cohomology of the chain complexes.

This discussion of right $G$-torsors carries over immediately to principal $G$-bundles.  The projection $\pi$ defines the vertical distribution $\TVP\subset TP$ of vectors along the fibers of $P$ by
\[ \TVP\= \Ker \pi_* \subset TP, \quad \TVP_{|\pi^{-1}x} \iso TP_x.\]
Translating \eqref{eq:torsorisos} into families of $G$-torsors, we obtain the natural $G$-equivariant isomorphisms
\begin{equation} \label{eq:verticaliso} \TVP \iso P \times \fg \iso \pi^* \fg_P,\end{equation}
where $\fg_P \overset{\pi}\to M$ is the Adjoint bundle $P\times_{Ad}\fg$.  The bracket in \eqref{eq:rightbracket} gives us the Lie bracket of right-invariant vertical vector fields on $P$
\[ \fg_P \times  \fg_P \overset{[\cdot, \cdot]}\longrightarrow \fg_P.\]
This induces a map of vector bundles
\[ d_\fg: \Lambda^i \fg_P^* \to \Lambda^{i+1} \fg_P^*\]
as defined in \eqref{eq:dg}.

\subsection{Horizontal Distribution}While the vertical distribution $\TVP$ requires no extra choices, there is no natural horizontal distribution $\THP$.  Instead, a choice of connection $\Theta$ on $P$ is equivalent to the equivariant choice of a horizontal distribution $\THP\subset TP$ such that
\[ \THP \oplus \TVP = TP.\]
In particular, if $\Theta \in \Omega^1(P;\fg)$ is the connection 1-form, then
\[ \THP \= \Ker \Theta.\]
This induces the $G$-equivariant isomorphism
\begin{equation}\label{eq:horizontaliso} \THP \iso \pi^* TM.\end{equation}

A connection $\Theta$ on $P$ gives a covariant derivative $\nabla$ on sections of any associated bundle.  Suppose $v \in \Gamma(M,\Lambda^k\fg_P^*)$.  One can take the derivative of a function whose values live in a fixed vector space, and $\widetilde{v} \= \pi^*v \in \Gamma(P,\Lambda^k \fg^*)^G$.  Given a vector $X\in T_xM$, the isomorphisms $\pi^*:T_xM \to \THP_p$ (defined by the connection) give a $G$-equivariant family of vectors $\widetilde{X}$ in $TP_{|\pi^{-1}(x)}$.  The derivative of $\widetilde{v}$ in the direction of $\widetilde{X}$ is well-defined, resulting in the equivariant $\widetilde{X}\widetilde{v} \in \Gamma(P,\Lambda^k \fg^*)^G$.  The isomorphism \eqref{eq:verticaliso} then defines
\begin{equation}\label{eq:covdiff} \widetilde{\nabla_X v} \= \widetilde{X} \widetilde{v} \in \Gamma(P,\Lambda^k \fg^*)^G\iso \Gamma(M, \Lambda^k \fg_P^*).\end{equation}
The operator $\nabla: \Gamma(M,\Lambda^k \fg_P^*)\to \Gamma (M,T^*M\otimes \Lambda^k \fg_P^*)$ is called the connection on the associated vector bundle and extends uniquely to a first order differential operator $d_\nabla$ on the complex
\begin{equation}\label{eq:dnabla} \Gamma(M;\Lambda^k \fg_P^*) \overset{\nabla}\longrightarrow \Omega^1(M;\Lambda^k \fg_P^*) \overset{d_\nabla}\longrightarrow\Omega^2(M;\Lambda^k \fg_P^*) \overset{d_\nabla}\longrightarrow\Omega^3(M;\Lambda^k \fg_P^*) \overset{d_\nabla}\longrightarrow \cdots \end{equation}
by requiring that $d_\nabla$ satisfy the Leibniz rule
\[ d_\nabla(\omega\otimes v) = (d\omega)\otimes v + (-1)^i \omega\otimes(\nabla v) \]
for all $\omega \in \Omega^i(M)$ and $v \in \Gamma (M,\Lambda^k \fg_P^*)$.

The possibility that $T^HP$ is not an integrable distribution leads to the notion of curvature, and the curvature 2-form $\Omega \in \Omega^2(P;\fg)$ is defined by the equation
\[ \Omega \= d\Theta + \frac{1}{2}[\Theta\wedge\Theta]. \]
The form $\Omega$ is only non-zero when evaluated on horizontal vectors.  If $X_H, Y_H \in \THP_p$, then
\begin{equation}\label{eq:curvature}\Omega(X_H, Y_H) = d\Theta(X_H, Y_H) = - \Theta([X_H, Y_H]).\end{equation}
Therefore, $\Omega\in \Omega^2(M, \fg_P)$, and we can form the contraction $\iota_\Omega$
\begin{equation}\label{eq:iotacurv} \Omega^i(M;\Lambda^j\fg^*_P) \overset{\iota_\Omega}\longrightarrow \Omega^{i+2}(M;\Lambda^{j-1}\fg^*_P).\end{equation}

\subsection{Bi-graded Complex}
The decomposition $TP = \THP \oplus \TVP$ induces a bigrading on forms
\[
 \Omega^{i,j}(P) \= \Gamma(P, \Lambda^i \THP \otimes \Lambda^j \TVP).
\]
The $G$-equivariant descriptions of $\TVP$ and $\THP$ in \eqref{eq:verticaliso} and \eqref{eq:horizontaliso}, respectively, imply that $G$-invariant forms on $P$ decompose as
\begin{align*}\label{eq:invariantforms} \Omega^k(P)^G &= \bigoplus_{i+j=k} \Omega^{i,j}(P)^G \\
& \iso \bigoplus_{i+j=k} \Gamma(M, \Lambda^i TM^* \otimes \Lambda^j \fg_P^*) \\
&=\bigoplus_{i+j=k} \Omega^i(M;\Lambda^j \fg_P^*).\end{align*}
Thus, an invariant $(i,j)$-form is naturally an $i$-form on $M$ with values in the bundle $\Lambda^k \fg_P^*$, where $\fg_P\to M$ is the Adjoint bundle of $P$.

The exterior derivative $d$ decomposes with respect to the distributions $\TVP$ and $\THP$.  In general,
\[ d= d^{-1, 2} + d^{0,1} + d^{1,0} + d^{2,-1},\]  
where $d^{a,b}: \Omega^{i,j}(P) \to \Omega^{i+a, j+b}(P)$, giving a bigraded cochain complex $\{ \Omega^{i,j}(P), d^{a,b}\}$.  When restricted to right-invariant forms, the complex $\{ \Omega^{i,j}(P), d^{a,b}\}$ has an explicit description using \eqref{eq:dg}, \eqref{eq:dnabla}, and \eqref{eq:iotacurv}.

\begin{prop}\label{prop:bigradedcomplex}Under the isomorphism $\{ \Omega^{i,j}(P)^G\} \overset{\pi^*}\leftarrow \{ \Omega^i(M, \Lambda^j \fg_P^*)\} $, the components of the exterior derivative $d$ are related by
\begin{itemize} \item $d^{-2,1} = 0$
\item $d^{0,1} \leftrightarrow   (-1)^i d_\fg $
\item $d^{1,0} \leftrightarrow d_\nabla$
\item $d^{2,-1} \leftrightarrow (-1)^i \iota_\Omega$
\end{itemize}
\end{prop}

We draw the bigraded complex isomorphic to $\{ \Omega^{i,j}(P)^G,d\}$ below:
{\scriptsize
\[ \hskip -.45in \xymatrix{ &\vdots &\vdots &\vdots &\vdots &\\
&\Omega^0(M;\Lambda^2\fg^*_P) \ar[u]_{d_\fg} \ar[r]^{\nabla} \ar@{.>}[drr]^>>>>{\iota_\Omega} & \Omega^1(M;\Lambda^2\fg^*_P) \ar[u]_{-d_\fg} \ar[r]^{d_\nabla} \ar@{.>}[drr]^>>>>{-\iota_\Omega}& \Omega^2(M;\Lambda^2\fg^*_P) \ar[u]_{d_\fg} \ar[r]^{d_\nabla} \ar@{.>}[drr]^>>>>{\iota_\Omega}& \Omega^3(M;\Lambda^2\fg^*_P) \ar[u]_{-d_\fg} \ar[r]^{d_\nabla}&\cdots\\
&\Omega^0(M;\fg^*_P)  \ar[u]_{d_\fg} \ar[r]^{\nabla} \ar@{.>}[drr]^>>>>{\iota_\Omega}& \Omega^1(M;\fg^*_P) \ar[u]_{-d_\fg} \ar[r]^{d_\nabla} \ar@{.>}[drr]^>>>>{-\iota_\Omega}& \Omega^2(M;\fg^*_P) \ar[u]_{d_\fg} \ar[r]^{d_\nabla} \ar@{.>}[drr]^>>>>{\iota_\Omega}& \Omega^3(M;\fg^*_P) \ar[u]_{-d_\fg} \ar[r]^{d_\nabla}&\cdots\\
&\Omega^0(M;\R) \ar[u]_{0} \ar[r]^{d_M} & \Omega^1(M;\R) \ar[u]_{0} \ar[r]^{d_M}& \Omega^2(M;\R) \ar[u]_{0} \ar[r]^{d_M}& \Omega^3(M;\R) \ar[u]_{0} \ar[r]^{d_M}&\cdots
} \] }
Note that above is not a commutative diagram, but summing over all possible paths between two points in the complex gives zero.  These relationships are explicitly given by $(d^{0,1} + d^{1,0} + d^{2,-1})^2=0$.

\begin{proof}[Proof of Proposition \ref{prop:bigradedcomplex}]
We will denote the pushforward of the standard exterior derivative by $\pi_*d$ and note that for any $\omega\otimes \psi \in \Omega^i(M;\Lambda^j\fg^*_P)$,
\[ (\pi_* d) (\omega\otimes \psi) \= (\pi^*)^{-1} d (\pi^* \omega\otimes \pi^* \psi).\]

The $d^{-2,1}=0$ because the distribution $\TVP$ is integral.  Each component of $d$ satisfies the Leibniz rule, so to describe $d$ on $\Omega^i(M;\Lambda^j\fg^*_P)$, it suffices to say how each component acts on $\Omega^i(M)$ and $\Gamma(M;\Lambda^j\fg^*_P)$.

For $(\pi_*d):\Omega^i(M;\R) \to \Omega^{i+1}(M;\R)\oplus\Omega^i(M; \fg^*_P)$, $\pi_*d = \pi_*d^{1,0}$ and is the exterior derivative on $M$, denoted $d_M$.  This follows from the naturality of $d$, 
\[ d\pi^*\Omega^i(M) = \pi^* d_M\Omega^i(M).\]

Let $\psi \in \Gamma(M, \Lambda^k \fg_P^*)$.  To calculate $\pi_*d^{0,1}$, let $X_0, \ldots, X_k \in \Gamma(M,\fg_P)$.  Tracing through definitions and \eqref{eq:dg} gives
\begin{align*}
d^{0,1} \pi^* \psi(\pi^*X_0, \ldots, \pi^*X_k) &= \sum_{i<j} \pi^*\psi ([\pi^*X_i, \pi^* X_k], \ldots, \widehat{\pi^*X_i}, \ldots, \widehat{\pi^* X_j}, \ldots)\\
&= \pi^* d_\fg \psi(X_0, \ldots, X_k).
\end{align*}
The Leibniz rule then implies that for $\omega \in \Omega^i(M), \psi \in \Gamma(M, \Lambda^j \fg_P^*)$,
\begin{align*}
d^{0,1} \pi^*(\omega\otimes \psi) &= (d^{0,1}\pi^*\omega) \wedge \pi^*\psi + (-1)^i (\pi^*\omega)\wedge (d^{0,1}\pi^*\psi) \\
&= (-1)^i \pi^*(\omega \otimes d_\fg \psi).
\end{align*}
Therefore, $\pi_*d^{0,1} = (-1)^i d_{\fg}$.  To calculate $\pi_*d^{1,0}$, let $X\in \Gamma(M,TM)$ and use \eqref{eq:covdiff}:
\begin{align*}
(d^{1,0}\pi^*\psi) (X) = (d \pi^*\psi)(X) = \widetilde{X} \widetilde{\psi} = \pi^* \nabla_X \psi.
\end{align*}
The Leibniz rule implies
\begin{align*}
 d^{1,0}\pi^*(\omega \otimes \psi) &= (d^{1,0}\pi^*\omega)\wedge (\pi^*\psi) + (-1)^i (\pi^*\omega)\wedge(d^{1,0}\pi^*\psi) \\
&= \pi^*(d_M\omega \otimes \psi) + (-1)^i \pi^*( \omega \otimes \nabla \psi) \\
&= \pi^* d_\nabla(\omega\otimes \psi).
\end{align*}
A similar use of \eqref{eq:curvature} shows that
\[ d^{2,-1}(\pi^*\psi) = \pi^*( \iota_\Omega \psi),\]
and hence $\pi_*d^{2,-1}$ is given by contracting along the $\fg_P$-valued 2-form $\Omega$.
\end{proof}

\subsection{Riemannian metric and the adjoint $d^*$}
We now construct a right-invariant Riemannian metric on $P$.  Let $g_M$ be a Riemannian metric on $M$, which in turn induces a metric on the distribution $\THP = \pi^*TM$.  Let $g_G$ be a bi-invariant metric on the Lie group $G$.  Such metrics $g_G$ exist due to the compactness of $G$ and are equivalent to $Ad$-invariant metrics on $\fg$.  Since $g_G$ is an $Ad$-invariant metric on $\fg$, it induces a right-invariant metric on $\TVP= \pi^*\fg_P$.  We also let $\Lambda^i g_M^*$ and $\Lambda^j g_g^*$ denote the induced metrics on $\Lambda^i TM^*$ and $\Lambda^j \fg_P^*$, respectively.  Using the connection $\Theta$ to give an orthogonal splitting $TP = \pi^* (TM \oplus \fg_P)$, we obtain a canonical metric $g_P$ on P by
\[ g_P \= \pi^*(g_M\oplus g_G).\]
More explicitly,
\[ g_P(X_1,X_2) \= g_M(\pi_*X_1, \pi_* X_2) + g_G(\Theta X_1, \Theta X_2).\]
The metric $g_P$ is evidently right-invariant, and we have the following isomorphism of cochain complexes with metric:
 \begin{align*}
 \{ \Omega^{i,j}(P), \pi^*( \Lambda^i g_M^* \otimes \Lambda^j g_G^*) \} ^G \overset{\iso}\longleftarrow \{ \Omega^i(M, \Lambda^j \fg_P^*), \Lambda^i g_M^* \otimes \Lambda^j g_G^* \}.
 \end{align*}

The adjoint $d^*:\Omega^k(P)\to \Omega^{k-1}(P)$ (with respect to $g_P$) commutes with isometries, and therefore, $d^*$ restricts to the right-invariant complex
\begin{equation}\label{eq:invtdualcomplex} \left( \Omega^i(M;\Lambda^j\fg^*_P),\pi_*d^* \right) \iso \left( \Omega^{i,j}(P)^G,d^{*g_P} \right) \into \left( \Omega^{i,j}(P),d^{*g_P} \right). \end{equation}
The adjoint $d^*$ decomposes under the bigrading,
\[ d^* = d^{0,1*} + d^{1,0*} + d^{2,-1*}, \]
where $d^{a,b*}:\Omega^{i,j}(P) \to \Omega^{i-a, j-b}(P)$.  Using our description of the differentials $d^{a,b}$ in Proposition \ref{prop:bigradedcomplex}, along with the right-invariance of $g_P$ and the compatibility of the metric with the bigrading, we see that on right-invariant forms
\[ d^{0,1*} = (-1)^i\pi^*(d^*_\fg),\quad d^{1,0*} = \pi^*(d^*_\nabla),\quad d^{2,-1*} = (-1)^i\pi^*(\iota^*_\Omega) .\]
where the adjoints are induced by the metrics $g$ and $g_G$ on $TM$ and $\fg_P$, respectively.

Therefore, the dual complex of right-invariant forms on $P$ is naturally isomorphic to the dual of the bigraded complex in Proposition \ref{prop:bigradedcomplex}.  In particular, the adjoint $d^*_M:\Omega^i(M;\R) \to \Omega^{i-1}(M;\R)$ is the usual adjoint with respect to the metric $g_M$, and the adjoint $d^*_\fg $ is induced from the usual adjoint \eqref{eq:liealgcoh} to the Lie algebra derivative with respect to the metric $g_G$.

Finally, we note that the complexes ultimately relevant are $\{ \Omega^*(P), d_\dd\}$ and $\{ \Omega^*(P), d^*_\dd\}$.  As noted in the definition \eqref{eq:ddelta},
\[ d_\dd = d^{0,1} + \dd d^{1,0} + \dd^2 d^{2,-1}, \quad d^*_\dd = d^{0,1*} + \dd d^{1,0*} + \dd^2 d^{2,-1}.\]
Therefore, the rescaled complexes of right-invariant forms are isomorphic to the following two complexes.  These two pictures will be very helpful when following the proofs in the next section.
{\scriptsize
\begin{equation}\label{eq:complex}\hskip -.45in \xymatrix{ &\vdots &\vdots &\vdots &\vdots&\\
&\Omega^0(M;\Lambda^2\fg^*_P) \ar[u]_{d_\fg} \ar[r]^{\dd \nabla} \ar@{.>}[drr]^>>>>{\dd^2 \iota_\Omega} & \Omega^1(M;\Lambda^2\fg^*_P) \ar[u]_{-d_\fg} \ar[r]^{\dd d^\nabla} \ar@{.>}[drr]^>>>>{-\dd^2 \iota_\Omega}& \Omega^2(M;\Lambda^2\fg^*_P) \ar[u]_{d_\fg} \ar[r]^{\dd d^\nabla} \ar@{.>}[drr]^>>>>{\dd^2 \iota_\Omega}& \Omega^3(M;\Lambda^2\fg^*_P) \ar[u]_{-d_\fg} \ar[r]^{\dd d^\nabla}&\cdots\\
&\Omega^0(M; \fg^*_P)  \ar[u]_{d_\fg} \ar[r]^{\dd \nabla} \ar@{.>}[drr]^>>>>{\dd^2 \iota_\Omega}& \Omega^1(M; \fg^*_P) \ar[u]_{-d_\fg} \ar[r]^{\dd d^\nabla} \ar@{.>}[drr]^>>>>{-\dd^2 \iota_\Omega}& \Omega^2(M; \fg^*_P) \ar[u]_{d_\fg} \ar[r]^{\dd d^\nabla} \ar@{.>}[drr]^>>>>{\dd^2 \iota_\Omega}& \Omega^3(M; \fg^*_P) \ar[u]_{-d_\fg} \ar[r]^{\dd d^\nabla}&\cdots \\
&\Omega^0(M;\R) \ar[u]_{0} \ar[r]^{\dd d_M} & \Omega^1(M;\R) \ar[u]_{0} \ar[r]^{\dd d_M}& \Omega^2(M;\R) \ar[u]_{0} \ar[r]^{\dd d_M}& \Omega^3(M;\R) \ar[u]_{0} \ar[r]^{\dd d_M} &\cdots
} \end{equation}
\begin{equation}\label{eq:dualcomplex}\hskip -.45in \xymatrix{  &\vdots \ar[d]^{d^*_\fg} &\vdots \ar[d]^{-d^*_\fg}&\vdots \ar[d]^{d^*_\fg} &\vdots \ar[d]^{-d^*_\fg} \\
 &\Omega^0(M;\Lambda^2\fg^*_P)  \ar[d]^{d^*_\fg} &\Omega^1(M;\Lambda^2\fg^*_P) \ar[l]_{\dd d^*_\nabla} \ar[d]^{-d^*_\fg} &\Omega^2(M;\Lambda^2\fg^*_P) \ar[l]_{\dd d^*_\nabla }\ar[d]^{d^*_\fg} \ar@{.>}[ull]^>>>>>>{\dd^2 \iota^*_\Omega}&\Omega^3(M;\Lambda^2\fg^*_P) \ar[l]_{\dd d^*_\nabla} \ar[d]^{-d^*_\fg} \ar@{.>}[ull]^>>>>>>{-\dd^2 \iota^*_\Omega}&\cdots \ar[l]_{\dd d^*_\nabla} \\
&\Omega^0(M; \fg^*_P)  \ar[d]^{0} &\Omega^1(M; \fg^*_P) \ar[l]_{\dd d^*_\nabla} \ar[d]^{0} &\Omega^2(M; \fg^*_P) \ar[l]_{\dd d^*_\nabla} \ar[d]^{0} \ar@{.>}[ull]^>>>>{\dd^2 \iota^*_\Omega} &\Omega^3(M; \fg^*_P) \ar[l]_{\dd d^*_\nabla} \ar[d]^{0} \ar@{.>}[ull]^>>>>{-\dd^2 \iota^*_\Omega}&\cdots \ar[l]_{\dd d^*_\nabla} \ar@{.>}[ull]^>>>>{\dd^2 \iota^*_\Omega}\\
&\Omega^0(M;\R)  &\Omega^1(M;\R) \ar[l]_{\dd d^*_M} &\Omega^2(M;\R) \ar[l]_{\dd d^*_M} \ar@{.>}[ull]^>>>>{\dd^2 \iota^*_\Omega}&\Omega^3(M;\R) \ar[l]_{\dd d^*_M}  \ar@{.>}[ull]^>>>>{-\dd^2 \iota^*_\Omega}&\cdots \ar[l]_{\dd d^*_M} \ar@{.>}[ull]^>>>>{\dd^2 \iota^*_\Omega}
}\end{equation}}
\section{Calculation of the Harmonic Forms}\label{section:harmonicforms}

Using the machinery and terminology described in Sections \ref{section:ss} and \ref{section:bigradedcomplex}, we proceed with the calculations.  To review notation, $G$ is a connected compact Lie group, and $P\to M$ is a principal $G$-bundle with connection $\Theta$ over the closed, oriented Riemannian manifold $(M,g_M)$.  We denote this information by $(M,P,g_M,\Theta)$.  It useful to remember that any bundle $P$ is a pullback of the universal bundle
\[ \xymatrix{ &P \ar^{\pi}[d] \ar^{f^*}[r] &EG \ar^{\pi}[d]\\ &M \ar^{f}[r] &BG.} \]
The harmonic $i$-forms on $M$, with respect to $g_M$, are denoted $\cH^i(M)$, and $\cH^j(G)$ denotes harmonic $j$-forms on $G$ with respect to a bi-invariant metric.  As defined earlier, and whose existence is made clear by Theorem \ref{thm:adiabaticharmonic},
\[ \cH^i(P) \= \lim_{\dd\to0} \Ker \Delta^i_{g_\dd} = \lim_{\dd\to 0} \cH^i_{g_\dd}(P) \subset \Omega^i(P) \]
is the finite-dimensional space of harmonic forms on $P$ in the adiabatic limit, and $\cH^i(P) \iso H^i(P;\R)$ canonically.

Proposition \ref{prop:bigradedcomplex} gives the natural isomorphism of bigraded cochain complexes
\[ \pi^*: \{ \Omega^i(M;\Lambda^j \fg_P^*)\} \overset{\iso}\longrightarrow \{\Omega^{i,j}(P)\}^G \subset \{ \Omega^{i,j}(P) \}. \]
In the following, we perform all caculations in terms of the left-hand complex (avoiding unnecessary use of $\pi^*$), but we state all major theorems in terms of $\Omega^*(P)$.  Finally, $\{E^{i,j}_K\}$ denotes the Leray--Serre spectral sequence, with $\R$-coefficients, for the fiber bundle $G\into P\to M$.

\begin{prop}\label{prop:basecoh}Given $(M,P,g_M,\Theta)$, assume that $E^{i,0}_2 = E^{i,0}_\infty$.  Then, 
\[ \pi^* \cH^i (M) \subset \cH^i(P).\]
\begin{proof}
Let $\omega \in \cH^i (M)$; i.e. $d_M\omega = d^*_M\omega=0$.  For dimensional reasons, the operators $d_\dd$ and $d^*_\dd$ take a slightly simpler form, and we calculate
\begin{align*} d_\dd \omega &= \dd d_M \omega = 0, \\
d^*_\dd \omega &=  \dd d^*_M \pm \dd^2 \iota_\Omega \omega = \pm \dd^2 \iota_\Omega \omega.
\end{align*}
Therefore, $\omega \in E^{i,0}_2 = E^{i,0}_\infty$.  By Corollary \ref{cor:Ei0toDelta}, $\omega \in \cH^i(P).$
\end{proof}\end{prop}

\subsection{1-forms.}We realize that this calculation of $\cH^1(P)$ is more complicated than it needs to be.  However, it provides a similar but simpler calculation than that of $\cH^3(P)$ and helps clarify the logical structure.  In examining the cohomology of the bundle $P$, we use the standard isomorphism 
\[ H^i(M;H^j(G;\R)) \iso H^i(M;\R^{\dim H^j(G;\R)}) \iso H^i(M;\R) \otimes H^j(G;\R).\]
(The action of $\pi_1(M)$ on $H^*(G;\R)$ is trivial.)  The connectedness of $M$ and $G$ imply that $H^0(M;\R)\iso H^0(G;\R)\iso \R$.  Furthermore, the classifying map gives a morphism between the spectral sequences for the universal bundle $EG\to BG$ and $P\to M$.  The relevant portion of the $E_2$ page is:
\begin{align}\label{eq:1dimss} 
\def\sseqgridstyle{\ssgridcrossword}
\sseqentrysize=1.4cm
\sseqxstep=0
\sseqystep=0
\begin{sseq}{3}{2} 
\ssdrop{R}
\ssmove 1 0 \ssdrop{0}
\ssmove 1 0 \ssdrop{H^2(BG)}
\ssmoveto 0 1 \ssdrop{H^1(G)} \ssarrow{2}{-1}
\ssmove {-1}{ 1} \ssdrop{} \ssdrop{} \ssdrop{}
\ssdroplabel[D]{d_2}
\end{sseq} 
\quad \text{\raisebox{1.4cm}{ $\overset{f^*}\Longrightarrow $} }\quad 
\def\sseqgridstyle{\ssgridcrossword}
\sseqentrysize=1.4cm
\sseqxstep=0
\sseqystep=0
\begin{sseq}{3}{2} 
\ssdrop{R}
\ssmove 1 0 \ssdrop{H^1(M)}
\ssmove 1 0 \ssdrop{H^2(M)}
\ssmoveto 0 1 \ssdrop{H^1(G)}
\ssarrow{2}{-1}
\ssmove {-1}{ 1} \ssdrop{} \ssdrop{} \ssdrop{}
\ssdroplabel[D]{d_2}
\end{sseq} 
\end{align} 
Examining the universal example, we know $EG$ is contractible, so $E_\infty^{i,j}=0$ for $(i,j)\neq(0,0)$.  This implies $d_2$ is an isomorphism between $H^1(G;\R)$ and the universal characteristic classes $H^2(BG;\R)$.  If $\phi \in H^1(G;\R)$, then let 
\[ \phi(EG)\= d_2\phi \in H^2(BG;\R)  \text{ and }\phi(P) \= f^*(\phi(EG)) \in H^2(M;\R). \]  By naturality of the spectral sequence, the class $f^*\phi = \phi \in H^1(G;\R)$, and 
\[ d_2\phi = d_2f^*\phi = f^*d_2\phi  =\phi(P) \in H^2(M;\R). \]
Therefore, the trangression $d_2$ sends an element in $H^1(G)$ to the corresponding characteristic class of the bundle $P$.  The only non-trivial differential in spectral sequence for $H^1(P;\R)$ is $d_2$, so we see that 
\[ E^1_\infty = E^{1,0}_2 \oplus \Ker d_2(E^{0,1}_2).\]
Therefore, $H^1(P;\R)$ sits in a short exact sequence
\begin{equation}\label{eq:h1P} 0 \to H^1(M;\R) \to H^1(P;\R) \to \Ker d_2 \to 0.\end{equation}

  Chern--Simons/Chern--Weil theory gives a more explicit geometric interpretation of this.  In a slight abuse of notation, let $\phi\in \fg^*$ be $Ad$-invariant, so $\phi$ is equivalent to an element of $\cH^1(G) \iso H^1(G;\R)$.  Via $d_2$, $\phi$ also determines an element in $H^2(BG;\R)$.  The choice of a connection $\Theta \in \Omega^1(P;\fg)$ constructs the Chern--Simons 1-form $\phi(\Theta)\in \Omega^1(P)$ and Chern--Weil 2-form $\phi(\Omega)\in \Omega^2(M)$, which satisfy the properties
\begin{itemize}
\item $i_x^* \phi(\Theta) = \phi(\theta) \in \cH^1(G),$
\item $R^*_g \phi(\Theta) = \phi(\Theta) \in \Omega^1(P),$
\item $d\phi(\Theta) = \phi (\Omega - \frac{1}{2}[\Theta \wedge \Theta]) = \phi(\Omega) \in \pi^*\Omega^2(M).$
\end{itemize}
In other words, $\phi(\Theta)$ is right-invariant, restricts to a canonical element in $\cH^1(G)$ on the fibers, and its derivative is the Chern--Weil form $\phi(\Omega)$.  Furthermore, $\phi(\Omega)$ is closed and its de Rham cohomology class is independent of the chosen connection $\Theta$.  Doing this carefully in the universal case implies that
\[ [\phi(\Omega)] = \phi(P)  \in H^2(M;\R).\]

To make this more concrete, consider the case of $G=U(n)$.  At the level of cohomology, $H^1(U(n);\R)\iso H^1(U(1);\R) \iso H^1(S^1;\R)$ is 1-dimensional and has a canonical generator.  The image, under $d_2$, of this canonical generator is denoted $c_1\in H^2(BU(n);\R)$, and $f^*c_1$ is the usual first Chern class $c_1(P)\in H^2(M;\R)$.  At the level of forms, let
\[ \phi(\bullet) = \frac{i}{2\pi}\tr(\bullet) \in u(n)^*.\]
This gives the Chern--Simons 1-form $\frac{i}{2\pi}\tr(\Theta) \in \Omega^1(P)$ and the first Chern form $\frac{i}{2\pi}\tr(\Omega) \in \pi^*\Omega^2(M)$.  It is important to note that if $c_1(P) =0 \in H^2(M;\R)$, the form $\frac{i}{2\pi}\tr(\Omega)$ is exact but not necessarily zero.

\begin{theorem}\label{theorem:1dim}Consider $(M,P,g_M,\Theta)$ where $G$ is a connected compact Lie group, and $\phi \in (\fg^*)^{Ad}$.  If $\phi(P) =0 \in H^2(M,\R)$, then
\[ \phi(\Theta) - \pi^*h \in \cH^1(P),\]
where $h\in \Omega^1(M)$ is the unique form such that $dh=\phi(\Omega)$ and $h\in d^*\Omega^2(M)$.  In fact, the form $\phi(\Theta)-\pi^*h$ is harmonic with respect to the metric $g_\dd$ for any $\dd >0$.
\begin{proof}
First, note that the vanishing of $\phi(P)\in H^2(M;\R)$ implies that the form $\phi(\Omega)\in \Omega^2(M)$ is exact.  The Hodge decomposition \eqref{eq:hodgedecomp} implies there exists a unique $h$ satisfying our assumptions.  We now proceed by explicit calculation using the bigraded complexes \eqref{eq:complex} and \eqref{eq:dualcomplex}.
\begin{align*}
d_\dd( \phi(\Theta) - \dd h) &= \rho_\dd d ( \phi(\Theta) - h) = \rho_\dd (\phi(\Omega) - \phi(\Omega)) = 0.
\end{align*}

For dimensional reasons, several terms in the $d^*_\dd$ calculation are automatically 0, leaving only:
\begin{align*}
d^*_\dd \phi(\Theta) &=\dd d^*_\fg \phi(\Theta) = 0, \\
d^*_\dd h &= \dd d^*_M h = 0.
\end{align*}
The top equation holds because $\phi(\Theta)$ restricts to a harmonic form on the fibers, and the bottom equation holds by the assumption $h\in d^*\Omega^2(M)$.  Therefore, $d^*_\dd ( \alpha(\Theta)-h)=0$.  This calculation can be visualized by using the complexes \eqref{eq:complex} and \eqref{eq:dualcomplex}.  Let dotted lines denote maps which send a element to zero, and let $\pm$ show when the images of two elements cancel.  This looks like:
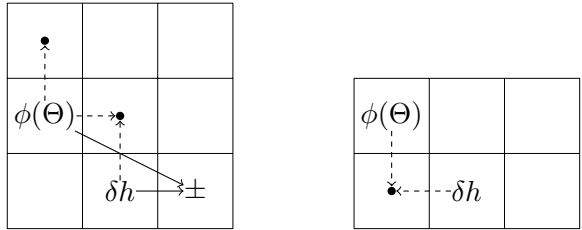
\begin{align*} \def\sseqgridstyle{\ssgridcrossword}
\sseqentrysize=1cm
\sseqxstep=0
\sseqystep=0
\begin{sseq}{3}{3}
\ssmoveto{0}{1} \ssdrop{\phi(\Theta)}
\ssmoveto{1}{0} \ssdrop{\delta h}
\ssmoveto{0}{1} \ssdashedarrow{0}{1} \ssdropbull
\ssmoveto{0}{1} \ssdashedarrow{1}{0} \ssdropbull
\ssmoveto{0}{1} \ssarrow{2}{-1} \ssdrop{\pm}
\ssmoveto{1}{0} \ssdashedarrow{0}{1}
\ssmoveto{1}{0} \ssarrow{1}{0}
\end{sseq}
\qquad \qquad
\def\sseqgridstyle{\ssgridcrossword}
\sseqentrysize=1cm
\sseqxstep=0
\sseqystep=0
\begin{sseq}{3}{2}
\ssmoveto{0}{1} \ssdrop{\phi(\Theta)}
\ssmoveto{1}{0} \ssdrop{\delta h}
\ssmoveto{0}{1} \ssdashedarrow{0}{-1} \ssdropbull
\ssmoveto{1}{0} \ssdashedarrow{-1}{0}
\end{sseq} \end{align*}

We have shown $\phi(\Theta) - \dd h \in \Ker L_\dd$.  Proposition \ref{prop:basecoh} implies that $E^{1,0}_\infty = \cH^1(M)$, and we have assumed that $h \in d^*\Omega^2(M)$, which is orthogonal to $\cH^1(M)$.  We can now use Proposition \ref{prop:EtoDelta}, and
\begin{align*} \rho_\dd^{-1} \left( \phi(\Theta) - \dd h \right) = \phi(\Theta) - h \in \cH^1(P). \end{align*}

Finally, we notice that we do not have to consider the adiabatic limit in order for $\phi(\Theta)-\pi^*h$ to be harmonic.  Since $d = \rho_\dd^{-1} d_\dd \rho_\dd$, our calculation above shows that
\begin{align*}
d(\phi(\Theta)-h) &= \rho_\dd^{-1} d_\dd (\alpha(\Theta) - \dd h) = 0, \\
d^*(\phi(\Theta)-h) &= \rho_\dd^{-1} d^*_\dd (\alpha(\Theta) - \dd h) = 0
\end{align*}
\end{proof}\end{theorem}

\begin{rem}The fact that $\phi(\Theta) - \pi^*h$ is harmonic for any $\dd >0$ is special to 1-forms.  In particular, explicit calculations of harmonic 3-forms on a bundle show that one must take an adiabatic limit for Theorem \ref{thm:harmonic3forms} to hold.
\end{rem}

Proposition \ref{prop:basecoh} and Theorem \ref{theorem:1dim} allow us to calculate $\cH^1(P)$ for any principal $G$-bundle $P$ when $G$ is compact connected.  Proposition \ref{prop:basecoh} gives representatives for the classes corresponding to $H^1(M;\R)$.  If we choose a basis for $\Ker d_2$, we use Theorem \ref{theorem:1dim} to obtain representatives of these classes.  Therefore, by the short exact sequence \eqref{eq:h1P}, these forms will together provide a basis for $\cH^1(P)$.  

\begin{cor}\label{cor:unitary1dim}Consider $(M,P, g_M,\Theta)$ where $G=U(n)$.  Then,
\[ \cH^1(P) = \begin{cases} \pi^* \cH^1(M) &\text{if } c_1(P) \neq 0 \in H^2(M;\R)\\
\pi^*\cH^1(M) \oplus \R[\csu - \pi^*h] &\text{if } c_1(P) = 0 \in H^2(M;\R)
\end{cases} \]
\begin{proof}
The term $E^{1,0}_2=E^{1,0}_\infty \iso H^1(M;\R)$, so Proposition \ref{prop:basecoh} implies that $\pi^*\cH^1(M)\subset \cH^1(P)$.  If $c_1(P) \neq 0$, then $\dim H^1(P;\R) = \dim H^1(M;\R)$, so we are done.

If $c_1(P)=0 \in H^2(M;\R)$, Theorem \ref{theorem:1dim} implies that $\csu - \pi^*h \in \cH^1(P)$.  The vector space spanned by it and $\pi^*\cH^1(M)$ has dimension equal to that of $H^1(P;\R)$, so we are done.
\end{proof}\end{cor}

\subsection{1-forms and 2-forms for $G$ semisimple.}Now, consider the case where $G$ is a compact, connected, semisimple Lie group.  The semisimplicity of $G$ implies that the decomposition of $\fg$ has no abelian factors, and hence $H^1(G;\R) = H^2(G;\R)=0$ \cite{CE48}.  The Leray--Serre spectral sequence then takes the form
{\scriptsize \[
\def\sseqgridstyle{\ssgridcrossword}
\sseqentrysize=1cm
\sseqxstep=0
\sseqystep=0
\begin{sseq}{5}{4} 
\ssdrop{R}
\ssmove 1 0 \ssdrop{0}
\ssmove 1 0 \ssdrop{0}
\ssmove 1 0 \ssdrop{0}
\ssmove 1 0 \ssdrop{H^4(BG)}
\ssmove{-4}{1} \ssdrop{0}
\ssmove 1 0 \ssdrop{0}
\ssmove 1 0 \ssdrop{0}
\ssmove 1 0 \ssdrop{0}
\ssmove 1 0 \ssdrop{0}
\ssmove{-4}{1} \ssdrop{0}
\ssmove 1 0 \ssdrop{0}
\ssmove 1 0 \ssdrop{0}
\ssmove 1 0 \ssdrop{0}
\ssmove 1 0 \ssdrop{0}
\ssmove{-4}{1} \ssdrop{H^3(G)}
\ssarrow{4}{-3}
\ssmove{-2}{2} \ssdroplabel[RD]{d_4}
\end{sseq} 
\quad
\text{\raisebox{2cm}{ $\overset{f^*}\Longrightarrow $} }
\quad
\def\sseqgridstyle{\ssgridcrossword}
\sseqentrysize=1cm
\sseqxstep=0
\sseqystep=0
\begin{sseq}{5}{4} 
\ssdrop{R}
\ssmove 1 0 \ssdrop{H^1(M)}
\ssmove 1 0 \ssdrop{H^2(M)}
\ssmove 1 0 \ssdrop{H^3(M)}
\ssmove 1 0 \ssdrop{H^4(M)}
\ssmove{-4}{1} \ssdrop{0}
\ssmove 1 0 \ssdrop{0}
\ssmove 1 0 \ssdrop{0}
\ssmove 1 0 \ssdrop{0}
\ssmove 1 0 \ssdrop{0}
\ssmove{-4}{1} \ssdrop{0}
\ssmove 1 0 \ssdrop{0}
\ssmove 1 0 \ssdrop{0}
\ssmove 1 0 \ssdrop{0}
\ssmove 1 0 \ssdrop{0}
\ssmove{-4}{1} \ssdrop{H^3(G)}
\ssarrow{4}{-3}
\ssmove{-2}{2} \ssdroplabel[RD]{d_4}
\end{sseq} 
\] }
\begin{prop}\label{prop:baseharm}Consider $(M,P,g_M, \Theta)$ where $G$ is a compact, connected, semisimple Lie group.  Then,
\begin{align*} \pi^* \cH^i(M) &= \cH^i(P) \text{ for } i=1,2. \\
\pi^* \cH^3(M) &\subset \cH^3(P).
\end{align*}
\begin{proof}
For $i=1,2,3$, $E^{i,0}_2 \iso H^i(M;\R)$ is not in the image of any non-trivial differential, implying that $\cH^i(M) = E^{i,0}_2=E^{i,0}_\infty$.  Proposition \ref{prop:basecoh} then implies $\pi^* \cH^i(M) \subset \cH^i(P)$ for $i=1,2,3$.  The spectral sequence also implies that for $i=1,2$, $H^i(P;\R) \iso H^i(M;\R)$, so the inclusion of $\pi^*\cH^i(M)$ spans $\cH^i(P)$.
\end{proof}\end{prop}

\subsection{3-forms for $G$ simple.}Analyzing the above spectral sequence for $G$ semisimple, we see that the contractibility of $EG$ implies 
\[ d_4:H^3(G;\R) \overset{\iso}\longrightarrow H^4(BG;\R).\]
Therefore, choosing generators of $H^3(G;\R)$ determines universal characteristic classes in $H^4(BG;\R)$, and $d_4$ maps $H^3(G;\R)$ to the pullback of these classes in any manifold $M$.  As $d_4: H^3(G;\R) \to H^4(M;\R)$ is the only non-trivial differential in the sequence calculating $H^3(P;\R)$, we have the short exact sequence
\begin{equation}\label{eq:kerd4} 0\to H^3(M;\R) \to H^3(P;\R) \to \Ker d_4 \to 0.\end{equation}

Just as in the discussion for $H^1(P;\R)$, this has a natural geometric interpretation from Chern--Simons and Chern--Weil theory.  If $\langle \cdot, \cdot \rangle$ is an $Ad$-invariant, symmetric, bilinear form on $\fg$, then $\langle \theta \wedge [\theta\wedge \theta] \rangle \in \Omega^3(G)$ is a bi-invariant 3-form and hence an element of $\cH^3(G)$.  Also, $\langle \cdot, \cdot \rangle$ naturally gives the Chern--Weil 4-form $\cwfour \in \Omega^4(M)$ whose de Rham class is the pullback of the universal class in $H^4(BG;\R)$.  The corresponding Chern--Simons 3-form for $\cwfour$ is 
\begin{equation}\label{eq:chernsimons3form} \csthree = \langle \Omega \wedge \Theta \rangle- \frac{1}{6} \langle \Theta \wedge [ \Theta \wedge \Theta] \rangle = \alpha^{2,1} + \alpha^{0,3} \in \Omega^3(P).\end{equation}
The notation $\alpha^{2,1} + \alpha^{0,3}$ denotes the decomposition of $\csthree$ with respect to the bigrading on $P$.  The Chern--Simons 3-form also satisfies \cite{Fre95}
\begin{itemize}
\item $i^*_x \csthree =  -\frac{1}{6}\langle \theta \wedge [\theta \wedge \theta] \rangle \in \cH^3(G)$,
\item $R^*_g\csthree = \csthree \in \Omega^3(P)$,
\item $d \csthree = \cwfour + \frac{1}{2} \langle \Omega \wedge [ \Theta \wedge \Theta] \rangle - \frac{1}{2} \langle \Omega \wedge [\Theta \wedge \Theta] \rangle = \cwfour \in \pi^*\Omega^4(M)$.
\end{itemize}
In other words, $\csthree$ is a right-invariant form that restricts to a canonical element in $\cH^3(G)$ on the fibers, and its derivative is $\cwfour$.

To make this more concrete, consider the case of $G=SU(n)$ for $n\geq 2$.  Then, $H^3(SU(n);\R) \iso H^3(SU(2);\R) \iso H^3(S^3;\R) \iso \R$, and there is a standard generator.  The image of this generator under $d_2:H^3(SU(n);\R)\to H^4(BSU(n);\R)$ is the universal second Chern class $c_2$.  The inner product $\langle \cdot, \cdot \rangle$ is a constant multiple of the matrix trace, giving 
\[c_2(P,\Omega) = \frac{1}{8\pi^2}\>\tr(\Omega \wedge \Omega).\]
Similarly, when $G=Spin(n)$ for $n=3$ or $n\geq 5$, $H^3(Spin(n);\R) \iso \R$ and there is a standard generator whose image is sent to the characteristic class $\phalf$ (which, when multiplied by 2, is the ordinary first Pontryagin class $p_1$.)  The inner product is again a constant multiple of the matrix trace.

While the above discussion holds for semisimple Lie groups, we now assume that $G$ is simple.  This implies that $H^3(G;\R) \iso \R$.  This is important, because it implies that if $\cwfour$ is exact, then $E_2^{0,3}=E_\infty^{0,3}$.  The following theorem uses that the sequence collapses at $N=2$.  If $H^3(G;\R)$ were 2-dimensional, and $\Ker d_4$ was 1-dimensional, the following proof would not suffice.

\begin{theorem}\label{thm:harmonic3forms}Consider $(M,P,g_M, \Theta)$ where $G$ is a simple Lie group.  Suppose $\cwfour$ is exact, so that the de Rham class is $0\in H^4(M;\R)$.  Then
\[ \alpha(\Theta) - \pi^*h \in \cH^3(P),\]
where $h\in \Omega^3(M)$ is the unique form such that $dh= \cwfour$ and $h\in d^*\Omega^4(M)$.
\begin{proof}
Supposing $\cwfour$ is exact implies, by Hodge decomposition \eqref{eq:hodgedecomp}, the existence of such a form $h$.  This also implies that $E^3_2 = E^3_\infty$.  In the calculations, we deal with
\[ \rho_\dd ( \alpha^{0,3} + \alpha^{2,1} - h^{3,0} ) = \alpha^{0,3} + \dd^2\alpha^{2,1} - \dd^3 h^{3,0}.\]
To use Proposition \ref{prop:EtoDelta}, we will show that $\alpha^{0,3} + \dd^2\alpha^{2,1} - \dd^3h$ extends formally to an element of $\Ker L_\dd$.  To do this, we first show that $\alpha^{0,3} + \dd^2 \alpha^{2,1}$ extends formally to $\Ker L_\dd$ by proving that for some $\psi$,
\begin{align*}
d_\dd ( &\alpha^{0,3} + \dd^2 \alpha^{2,1} + \dd^3 \psi ) \in \dd^4 \Omega^4(P)[\dd], \\
d^*_\dd ( &\alpha^{0,3} + \dd^2 \alpha^{2,1} + \dd^3 \psi ) \in \dd^4 \Omega^2(P)[\dd],
\end{align*}
and then applying Proposition \ref{prop:basis}.  Then, we show the term $\psi^{0,3}$ must be $h$ and apply Proposition \ref{prop:EtoDelta}.

First, we see that
\begin{align*} d_\dd \left( \alpha^{0,3} + \dd^2 \alpha^{2,1} - \dd^3 h^{3,0} \right) &= \rho_\dd d \left( \csthree - h \right) \\
&= \rho_\dd  \left( \cwfour - \cwfour\right) = 0, \end{align*}
but there is a non-trivial term of order $\dd^3$ in
\begin{align*}
d^*( \alpha^{0,3} + \dd^2 \alpha^{2,1} - \dd^3 h^{3,0}) &= d^*_\fg \alpha^{0,3} + \dd^3 d^*_\nabla \alpha^{2,1} - \dd^3 d^*_M h^{3,0} + O(\dd^4) \\
&= \dd^3 d^*_\nabla \alpha^{2,1} + O(\dd^4).
\end{align*}
We listed all pertinent non-trivial components in the $d^*_\dd$ calculation above.  All others are either trivially zero (for dimensional reasons) or have order greater than $\dd^3$.  Because $\alpha^{0,3}$ is harmonic on fibers, $d^*_\fg \alpha^{0,3} =0$.  Likewise, we assume $h$ coexact as a form in $M$, which implies $d^*_M h=0$.

We now must cancel the term
\[ d^*_\nabla \alpha^{2,1} \in \Omega^1(M; \fg_P^*).\]
Fortunately, $H^1(\fg) = 0$, so the Hodge decomposition implies $\fg_P^* = \text{coexact}(\fg_P^*)$.  Thus, $d^*_\fg: \text{exact}(\Lambda^2 \fg_P^*) \to \fg^*_P$ is an isomorphism, and we define
\begin{equation}\label{eq:betaterm}
\beta^{1,2} \= (d^*_\fg)^{-1} \left( d^*_\nabla \alpha^{2,1} \right) .
\end{equation}
The details of this are discussed in the following Lemma \ref{lem:betaterm}.  The important properties of $\beta^{1,2}$ are that
\[ d^*_\fg \beta^{1,2} = d^*_\nabla \alpha^{2,1}, \quad d_\fg \beta^{1,2} =0.\]

We now calculate that
\begin{align*}\label{eq:dorderfour} d_\dd \left( \alpha^{0,3} + \dd^2 \alpha^{2,1} - \dd^3 h - \dd^3 \beta^{1,2} \right) &= - \dd^3 d_\dd \beta^{1,2} \\
& =- \dd^3 d_\fg \beta^{1,2} + O(\dd^4) = 0 +O(\dd^4)
\\ \ \\
d^*_\dd \left( \alpha^{0,3} + \dd^2 \alpha^{2,1} - \dd^3 h^{3,0} - \dd^3 \beta^{1,2} \right) &= 
\dd^3 \left( d^*_\nabla \alpha^{2,1} - d^*_\fg \beta^{1,2} \right) + O(\dd^4) \\
&= 0+ O(\dd^4).
\end{align*}
The structure of the two above calculations can be visualized by using the complexes \eqref{eq:complex} and \eqref{eq:dualcomplex}.  This is shown below by letting dotted lines denote maps which send an element to zero, and letting $\pm$ show when the images of two elements cancel.  All elements of order greater than $\dd^3$ are ignored.
\[ \def\sseqgridstyle{\ssgridcrossword}
\sseqentrysize=1cm
\sseqxstep=0
\sseqystep=0
\begin{sseq}{5}{5} 
\ssmoveto{0}{3} \ssdrop{\alpha^{0,3}}
\ssmove{1}{-1} \ssdrop{\delta^3\beta^{1,2}}
\ssmove{1}{-1} \ssdrop{\dd^2\alpha^{2,1}}
\ssmove{1}{-1} \ssdrop{-\dd^3h^{3,0}}
\ssmoveto{0}{3} \ssdashedarrow{0}{1} \ssdropbull
\ssmoveto{0}{3} \ssdashedarrow{1}{0} \ssdropbull
\ssmoveto{0}{3} \ssarrow{2}{-1} \ssdrop{\pm}
\ssmoveto{2}{1} \ssarrow{0}{1}
\ssmoveto{2}{1} \ssdashedarrow{1}{0} \ssdropbull
\ssmoveto{2}{1} \ssarrow{2}{-1} \ssdrop{\pm}
\ssmoveto{3}{0} \ssdashedarrow{0}{1}
\ssmoveto{3}{0} \ssarrow{1}{0}
\ssmoveto{1}{2} \ssdashedarrow{0}{1}
\end{sseq} 
\quad
\quad
\def\sseqgridstyle{\ssgridcrossword}
\sseqentrysize=1cm
\sseqxstep=0
\sseqystep=0
\begin{sseq}{5}{4} 
\ssmoveto{0}{3} \ssdrop{\alpha^{0,3}}
\ssmove{1}{-1} \ssdrop{\delta^3\beta^{1,2}}
\ssmove{1}{-1} \ssdrop{\dd^2\alpha^{2,1}}
\ssmove{1}{-1} \ssdrop{-\dd^3h^{3,0}}
\ssmoveto{0}{3} \ssdashedarrow{0}{-1} \ssdropbull
\ssmoveto{1}{2} \ssarrow{0}{-1} \ssdrop{\pm}
\ssmoveto{2}{1} \ssdashedarrow{0}{-1} \ssdropbull
\ssmoveto{2}{1} \ssarrow{-1}{0}
\end{sseq} \]

We note that $\alpha^{2,1} \in E^\perp_\infty$ for bi-grading reasons, since the spectral sequence satisfies
\[ E^3_\infty = E^{3,0}_\infty \oplus E^{0,3}_\infty.\]
Proposition \ref{prop:basis} therefore implies the existence of a power series
\[ \omega_\dd = \alpha^{0,3} + \dd^2 \alpha^{2,1} + \dd^3 \omega_3 + \cdots \in \Omega^3(P)\]
for some $\omega_3, \omega_4, \ldots \in E_\infty^\perp$, such that formally $d_\dd \omega_\dd = d^*_\dd \omega_\dd = 0$.

Now, investigate the term $\omega_3^{3,0}$ in $\omega_\dd$.  The equation 
\[ 0 = d_\dd \omega_\dd = \dd^4 \left( \cwfour + d_M \omega_3^{3,0} \right) + O(\dd^5) \]
implies that $d_M \omega_3^{3,0} =- \cwfour$.  The equation $d^*_\dd \omega_\dd =0$, when restricted to the $(2,0)$ component, shows that
\[ d^*_M \omega_3^{3,0} = 0.\]
For $\omega_3$ to be an element of $E^\perp_\infty$, $\omega_3$ must be orthogonal to $E^{3,0}_\infty = \cH^3(M)$ (calculated in Proposition \ref{prop:baseharm}).  Therefore, $\omega_3^{3,0} =-h,$ where $h \in \Omega^3(M)$ is the unique form such that 
\[ d_h =  \cwfour, \quad d^* h = 0, \quad h \perp \cH^3(M).\]
Therefore, there exists a formal $L_\dd$-harmonic power series $\omega_\dd$ is of the form
\[ \omega_\dd = \alpha^{0,3} + \dd^2 \alpha^{2,1} - \dd^3 h^{3,0} + \dd^3( \omega_3 - h^{3,0}) + \dd^4 \omega_4 + \cdots.\]

Finally,
\begin{align*} \rho_\dd ^{-1}&\left( \alpha^{0,3} + \dd^2 \alpha^{2,1} + \dd^3 h^{3,0} + \dd^3(\omega_3-h^{3,0} ) + O(\dd^4) \right) \\
&= \alpha^{0,3} + \alpha^{2,1} - h^{3,0} + O(\dd).
\end{align*}
The polynomial satisfies the conditions of Proposition \ref{prop:EtoDelta}, so the constant term in $\rho_\dd^{-1}(\alpha^{0,3}  +\dd^2\alpha^{2,1} -\dd^3h + \cdots)$ is the adiabatic limit of a harmonic form on $P$:
\[ \alpha^{0,3} + \alpha^{2,1} - h^{3,0} \in \cH^3(P).\]
\end{proof}\end{theorem}

\begin{cor}\label{cor:harmonic3forms}Consider $(M,P,g_M, \Theta)$ where $G=SU(n)$ for $n\geq 2$.  Then, letting $c_2$ be the standard generator of the 1-dimensional space $H^4(BSU(n);\R)$,
\[ \cH^3(P) = \begin{cases} \pi^* \cH^3(M) & \text{ if } c_2 (P) \neq 0 \in H^4(M;\R)\\
\pi^*\cH^3(M) \oplus \R[\csthree - \pi^*h] & \text{ if } c_2 (P) =0 \in H^4(M;\R)
\end{cases}\]
where $h\in \Omega^3(M)$ is the unique form such that $dh= \cwfour$ and $h\in d^*\Omega^4(M)$.

When $G=Spin(n)$ for $n=3$ or $n\geq 5$, the above statement holds with $c_2$ replaced by $\phalf$.
\begin{proof}The proof is exactly the same as Corollary \ref{cor:unitary1dim}.  We combine Proposition \ref{prop:baseharm} with Theorem \ref{thm:harmonic3forms} and the exact sequence
\[ 0\to H^3(M;\R) \to H^3(P;\R) \to H^3(G;\R) \overset{d_4}\to H^4(M;\R),\]
noting that $\Ker d_4$ is either 0 or 1-dimensional, depending solely on whether the form $\cwfour$ is exact.
\end{proof}\end{cor}

\begin{lemma}\label{lem:betaterm}For any $\psi \in \Omega^i(M;\fg_P^*)$ there exists $\beta \in \Omega^i(M;\Lambda^2 \fg_P^*)$ such that
\[ d_\fg \beta = 0, \quad d^*_\fg \beta = \psi.\]
\begin{proof}
This lemma is just the Green's function for the Laplacian on $G$.  As noted earlier, a semisimple $G$ has $H^1(\fg) = 0$.  Therefore, Harm$(\fg^*) =0$, and the Hodge decomposition of
\[ \xymatrix{ &0 \ar[r] &\R \ar[r]^0\ar@/_1pc/[l]  &\fg^* \ar[r]^{d_\fg}\ar@/_1pc/[l]_{0} &\Lambda^2\fg^* \ar[r]^{d_\fg} \ar@/_1pc/[l]_{d_\fg^*} &\cdots \ar@/_1pc/[l]_{d_\fg^*} , } \]
gives 
\[ \fg^* = \text{Harm}(\fg^*)\oplus \text{exact}(\fg^*)\oplus \text{coexact}(\fg^*) = \text{coexact}(\fg^*).\]
The map
\[ d^*_\fg: \text{exact}(\Lambda^2 \fg^*) \to \fg^* \]
is an isomorphism.  Hence, we have a bundle morphism
\[ (d^*_\fg)^{-1}: \fg_P^* \to \Lambda^2 \fg^*_P .\]
The image of $(d^*_\fg)^{-1}$ is exact$(\Lambda^2\fg^*_P)$, so $d_\fg \o (d^*_\fg)^{-1} = 0$.  Applying $(d^*_\fg)^{-1}$ to any $\psi \in \Omega^i(M;\fg_P^*)$ gives the desired form.
\end{proof}\end{lemma}


\bibliographystyle{alphanum}
\bibliography{mybibliography}

\end{document}